\documentclass[a4paper,10pt]{amsart}

\usepackage{amssymb}
\usepackage{fullpage}
\usepackage{enumerate}
\usepackage[OT4]{fontenc}

\theoremstyle{definition}
\newtheorem{definition}{Definition}
\newtheorem{example}[definition]{Example}
\theoremstyle{plain}
\newtheorem{theorem}[definition]{Theorem}
\newtheorem*{theorem*}{Main Theorem}

\theoremstyle{remark}
\newtheorem{remark}[definition]{Remark}
\newtheorem{problem}[definition]{Problem}

\DeclareMathOperator{\Ass}{Ass}

\DeclareMathOperator{\reg}{reg}
\DeclareMathOperator{\HF}{HF}
\DeclareMathOperator{\HP}{HP}
\DeclareMathOperator{\aHP}{aHP}
\DeclareMathOperator{\aHF}{aHF}
\DeclareMathOperator{\gin}{gin}

\DeclareMathOperator{\vol}{vol}
\DeclareMathOperator{\conv}{conv}

\DeclareMathOperator{\dpth}{depth}

\def\field{\mathbb{K}}
\def\PP{\mathbb{P}}
\def\RR{\mathbb{R}}
\def\TT{\mathbb{T}}

\def\muh{\widehat{\mu}}

\let\to\longrightarrow

\begin{document}

\title{Asymptotic Hilbert Polynomial and a bound for Waldschmidt constants}

\author{Marcin Dumnicki, {\L}ucja Farnik, Halszka Tutaj-Gasi\'nska}

\thanks{Corresponding author: {\L}ucja Farnik.}
\keywords{symbolic powers, asymptotic invariants}

\subjclass[2010]{13P10; 14N20}

\begin{abstract}

In the paper we give an upper bound for the Waldschmidt constants
of the wide class of ideals. This generalizes the result obtained
by Dumnicki, Harbourne, Szemberg and Tutaj-Gasi{\'n}ska, Adv. Math.
2014, \cite{DHST}. Our bound is given by a root of a suitable derivative of a certain polynomial associated with the asymptotic Hilbert polynomial.
\end{abstract}

\maketitle

\section{Introduction}

In the recent years the asymptotic invariants of ideals  have
aroused great interest and have been studied by many researchers,
see for example \cite{ELMNP1}, \cite{ELMNP2}, \cite{DSST},
\cite{DST}, \cite{Mayci}, \cite{Maypoints}, \cite{Mus} and others.
 One of
these asymptotic invariants is the so called Waldschmidt constant
of an ideal $I\in\mathbb{K}[\mathbb{P}^n]$ denoted by
$\widehat\alpha(I)$, see for example \cite{BH}, \cite{BocHar10b},
\cite{BCH14}. The constant is the limit of a sequence of quotients
of the initial degrees of the $m$-th symbolic power of the ideal
by $m$ (see Definition \ref{defWald}). Computing this constant is
a hard task in general, as it is difficult to compute the initial
degree of a symbolic power of an ideal. For example, if $I$ is the
ideal of $s$ points in $\mathbb{P}^2$ in generic position, finding
$\widehat\alpha(I)$ means finding the Seshadri constant of these
points. The Seshadri constant is defined as the infimum of the
quotients $\frac{\textrm{deg} C}{m_1+\ldots+m_s}$, where $C$ is a
curve passing through $P_1,\ldots,P_s $ with multiplicities
$m_1,\ldots,m_s$.  By the famous Nagata conjecture (see eg
\cite{DHST} or \cite{primer} and the references therein) we expect
in this situation the equality $\widehat\alpha(I)=\sqrt{s}$, for
$s\geq 10$. Here we know that $\widehat\alpha(I)\leq \sqrt{s}$,
but in general we have no bounds on $\widehat\alpha(I)$ at all.
The result of Esnault and Viehweg in \cite{EV} gives a lower bound
of the Waldschmidt constant of an ideal of distinct points in~$\mathbb{P}^n$. 

In \cite{DHST} the authors give an upper bound of
$\widehat\alpha(I)$ in case $I$ is an ideal of a sum of disjoint
linear subspaces of $\mathbb{P}^n$, see Theorem
\ref{zpracyzBrianem}. In the present paper we generalize this
result and give an upper bound of $\widehat\alpha(I)$ for a wide
class of ideals (namely, radical ideals with linearly bounded
regularity of symbolic powers, see Preliminaries for the
definitions). To find this bound we use $\aHP_I(t)$, the so called
asymptotic Hilbert polynomial of $I$, defined in \cite{DST}. The
bound is given by the root of a suitable differential of the
polynomial $\Lambda_I(t):=\frac{t^n}{n!}-\aHP_I(t)$. The main
result of the present paper is the following theorem:
\begin{theorem*}
Let $I$ be a radical  homogeneous ideal in $\mathbb{K}[\mathbb{P}^n]$ with linearly bounded
regularity of symbolic powers. Assume that in the sequence
$\left\{\dpth I^{(m)}\right\}$ there exists a constant subsequence
of value $n-c$. Then
$$\Lambda^{(c)}_I(\widehat{\alpha}(I))\leq 0,$$ where
$\Lambda^{(c)}_I$ denotes the $c$-th derivative of $\Lambda_I$.

In particular $\widehat{\alpha}(I)\leq \gamma_{\Lambda_I^{(c)}}$,
where $\gamma_{\Lambda_I^{(c)}}$ is the largest real root of the
polynomial $\Lambda_I^{(c)}(t)$.
\end{theorem*}

The paper is organized as follows. In the second section we recall
the necessary notions, in the third we prove the main result. The
fourth section contains some interesting examples. In particular,
in Example~\ref{pochodna} we show  that it is necessary to take the root of a
derivative of the polynomial $\Lambda_I$, not of the polynomial itself, as
$\Lambda_I(\widehat{\alpha}(I))>0$. Example \ref{five-crosses}
shows that  we may get worse bounds on $\widehat{\alpha}(I)$ by
computer-aided computations than by application of  the Main
Theorem.

 \section{Preliminaries}

In the paper \cite{DST} the authors define the so called
asymptotic Hilbert function and asymptotic Hilbert polynomial.
Namely, let $\field$ be an algebraically closed field of
characteristic zero, by $\field[\PP^n]=\field[x_0,\dots,x_{n}]$ we
denote the homogeneous coordinate ring of the projective space
$\PP^n$. Let $I$ be a homogeneous radical ideal in
$\field[\PP^n]$, let $I^{(m)}$ denote its $m$-th symbolic power,
defined as:
$$I^{(m)} = \field[\PP^n] \cap \bigcap_{Q \in \Ass(I)} (I^m)_{Q},$$
where localizations are embedded in a field of fractions of $\field[\PP^n]$ (\cite{Eis95}). By the Zariski-Nagata theorem, for a radical homogeneous ideal $I$
in a polynomial ring over an algebraically closed field, the $m$-th symbolic power $I^{(m)}$ is equal to
$$I^{(m)} = \bigcap_{p \in V(I)} \mathfrak{m}_{p}^m,$$
where $\mathfrak{m}_p$ denotes the maximal ideal of a point $p$, and $V(I)$ denotes the set of zeroes of $I$. In characteristic zero, the symbolic power (of a
radical ideal) can also be described as the set of polynomials which vanish to order $m$ along $V(I)$; this (compare \cite{Sul08}) can be written as:
$$I^{(m)} = \left(f : \frac{\partial^{|\alpha|} f}{\partial x^{\alpha}} \in I \text{ for } |\alpha| \leq m-1 \right).$$

To define the asymptotic Hilbert polynomial, recall that
the Hilbert function $\HF_{I}$ of a homogeneous ideal $I$ is defined as
$$\HF_{I}(t) = \dim_{\field}(\field[\PP^n]_t/I_t).$$
For $t$ big enough the above function behaves as a polynomial, the Hilbert polynomial $\HP_{I}$ of $I$.

Let us define ideals with linearly bounded symbolic regularity (ie satisfying LBSR condition):
\begin{definition}\label{LBSR}
Let $I$ be a homogenous ideal in $\field(\PP^n)$. We say that $I$ \emph{satisfies linearly bounded symbolic regularity}, or \emph{is LBSR} for short, if there exist
constants $a,b>0$ such that
$$\reg\left(I^{(m)}\right)\leq am+b.$$
\end{definition}

It is worth observing, that  we do not know, so far, any example of a homogeneous ideal, which is not LBSR. The list of ideals which are proved  to be LBSR
may be found eg in \cite{DST}.

Now we recall the definitions of the asymptotic Hilbert function and the asymptotic Hilbert polynomial of an ideal $I$.
\begin{definition}
The \emph{asymptotic Hilbert function of $I$} is
$$\aHF_I(t):=\lim_{m\to \infty}\frac{\HF_{I^{(m)}}(mt)}{m^n}$$
in case that the limit exists.
\end{definition}
 In \cite{DST} it is shown that if $I$ is a radical ideal then the limit exists.

\begin{definition}\label{defaHP}
The \emph{asymptotic Hilbert polynomial of $I$} is
$$\aHP_I(t):=\lim_{m\to \infty}\frac{\HP_{I^{(m)}}(mt)}{m^n}$$
in case that the limit exists.
\end{definition}
In \cite{DST} it is shown that if $I$ is a radical LBSR ideal  then the limit exists.

Recall the definition of the Waldschmidt constant of an ideal $I$:

\begin{definition}\label{defWald}
$$\widehat\alpha(I):=\lim_{m\to\infty}\frac{\alpha(I^{(m)})}{m}=\inf_{m\to\infty}\frac{\alpha(I^{(m)})}{m},$$
where $\alpha(J)$ is the least degree of a nonzero polynomial
appearing in $J$ (called the initial degree of $J$).
\end{definition}

In \cite{DHST} the authors defined a polynomial
$\Lambda_{n,r,s}(t)$, namely let $L$ be a sum of $s$ disjoint
linear subspaces of dimension $r$ in $\mathbb{P}^n$ (called a
\emph{flat}; by a \emph{fat flat} we denote such subspaces with
multiplicities), as in \cite{DHST}. Let $I$ be the ideal of $L$.
Define
$$P_{n,r,s,m}(t):=\binom{t+n}{n}-\HP_{I^{(m)}}(t).$$
Substitute $t$ by $mt$ into $P_{n,r,m,s}(t)$ and regard it as a
polynomial in $m$ (this is indeed a polynomial, see~\cite{DHST}).
The leading term of this polynomial is  denoted by
$\Lambda_{n,r,s}(t).$

In \cite{DST} it is shown that
$$\aHP_I(t)=\frac{t^n}{n!}-\Lambda_{n,r,s}(t),$$
where $I$ is the ideal of the fat flat.

In this paper we define $\Lambda_I(t)$ for any radical LBSR ideal as
$$ \Lambda_I(t)=\frac{t^n}{n!}-\aHP_I(t).$$

The main result of our paper gives an upper bound for $\widehat\alpha(I)$
in terms of the largest root of a suitable derivative of $\Lambda_I$.

\section{Main Result}

The main result of our paper is the theorem below, giving an upper
bound for $\widehat\alpha(I)$, where $I$ is radical and
satisfies LBSR condition. Thus, this theorem generalizes the
Theorem 2.5 from \cite{DHST}, where the bound is proved for ideals
of linear subspaces only. More comments on this generalization are
in Remark \ref{o-fat-flatsach}.

\begin{theorem}[Main Theorem]\label{main}
Let $I$ be a radical 
homogeneous LBSR ideal.
Assume that in the sequence $\left\{\dpth I^{(m)}\right\}$ there
exists a constant subsequence of value $n-c$. Then
$$\Lambda^{(c)}_I(\widehat{\alpha}(I))\leq 0,$$ where
$\Lambda^{(c)}_I$ denotes the $c$-th derivative of $\Lambda_I$.
\end{theorem}

\begin{proof}

\textbf{Case $c=0$.} In this case there is a subsequence of depth
$n$.

We have to prove  $\Lambda_I(\widehat\alpha(I))\leq 0$, ie that

$$\lim_{m\to \infty}\frac{\binom{n+mt}{n}-\HP_{I^{(m)}}(mt)}{m^n}$$
is less than or equal to zero for $t=\widehat\alpha(I)$.

Recall that
$$\widehat\alpha(I)=\lim_{k\to\infty}\frac{\alpha(I^{(k)})}{k}.$$
Take
$$t_m:=\frac{\alpha(I^{(m)})-1}{m}.$$

If we prove that

$$\Lambda_I(t_m)=\lim_{m\to \infty}\frac{\binom{n+mt_m}{n}-\HP_{I^{(m)}}(mt_m)}{m^n}\leq 0,$$
then we are done  as $\Lambda_I$ is continuous. 

 Observe that
$$ \binom{n+mt_m}{n}-\HF_{I^{(m)}}(mt_m)=\binom{n+\alpha(I^{(m)})-1}{n}-\HF_{I^{(m)}}(\alpha(I^{(m)})-1)= 0$$
from the definition of $\alpha$.

So, it is enough to show that in our case
\begin{equation}\label{eq:HFleqHP}
\HF_{I^{(m)}}(t) \leq \HP_{I^{(m)}}(t)
\end{equation}
for all nonnegative integers $t \geq \alpha({I^{(m)}})-1$.

For this, observe that for any ideal $J$ in $\field[\PP^n]$, we
have that $\HF_J=\HF_{\gin(J)}$, and $\HP_J=\HP_{\gin(J)}$  (where
$\gin(J)$ is the initial ideal of $J$, with respect the degree reverse
lexicographical order, of a generic coordinate change of $J$, cf
the beginning of Section 4). 

Ideal $I$ is radical thus saturated. Since $\dpth \gin(I) = \dpth I=n$, 
by Lemma 3.1 from \cite{HerSri}  $\gin(I)$ involves all but one of the variables. Hence it is enough to prove
the claim \eqref{eq:HFleqHP}  for such ideals.

Fix $m$. Let $K=I^{(m)}$. From now on we assume that $K$ is
monomial. Let $M(n)$ denote the set of
monomials in variables $x_0,\dots,x_n$. Observe that
$$\{ \mu \in M(n) : \mu \in K, \, \deg \mu = t \} = \{ \mu \in M(n-1) : \mu \in K, \, \deg \mu \leq t \}.$$
In other words,
$$\HF_{I}(t) = \# \{ \mu \in M(n-1) : \deg \mu \leq t, \, \mu \notin K \}.$$

 Let $K = (\mu_1,\dots,\mu_k)$ where $\mu_1,\dots,\mu_k$ are monomial generators of $K$. Let $\muh := \gcd(\mu_1,\dots,\mu_k)$, let
$J=(\muh)$ be an ideal generated by $\muh$, and let
$$\Delta := \{ \mu \in M(n-1) : \mu \in J \setminus K\}.$$
We will show by induction on $n$ that $\Delta$ is finite. For $n=1$ the claim is obvious (since $K=J$ is this case). So let $n$ be arbitrary. By $\deg_j(\mu)$
we denote the degree of $\mu$ with respect to $x_j$. For each $j=0,\dots,n-1$ let
$$d_j := \max_i \{ \deg_j(\mu_i) \}, \qquad \Delta_j := \Delta \cap \{\mu : \deg_j(\mu) \geq d_j\}.$$
Observe that each $d_j$ is defined, as each variable appears as a factor in some generator. Let $K_j$ be the dehomogenization of $K$ with respect to $x_j$, let
$\muh_j$ be the dehomogenization of $\muh$. Observe that the dehomogenization of each $\mu \in \Delta_j$ belongs to the set $\Delta$ defined for $K_j$ and
$\muh_j$ (the assumption that $\deg_j(\mu) \geq d_j$ plays a crucial role here). Hence, by inductive assumption (and since the degree of $\mu$ is bounded) each
set $\Delta_j$ is finite. Observe also that
$$\Delta \subset \bigcup_{j=0}^{n-1} \Delta_j \cup (\Delta \cap \{ \mu : \deg_j(\mu) \leq d_j \text{ for } j=0,\dots,n-1\}).$$
Each of these sets is finite, hence $\Delta$ is finite, as claimed.

From finiteness of $\Delta$ is follows that, for $t$ big enough (it is enough to take $t$ bigger than the maximum degree of a monomial in $\Delta$),
$$\HF_K(t) = \HF_J(t) + \# \Delta.$$
The same holds for Hilbert polynomials for all $t$. Observe that: 
\begin{equation}\label{eq:HF_J}
\HF_J(t) = \begin{cases}
\binom{n+t}{n} - \binom{n+t-\deg(\muh)}{n} & \text{ for } t \geq \deg(\muh)=\alpha(J), \\
\binom{n+t}{n} & \text{ for } t < \deg(\muh).
\end{cases}
\end{equation}
Therefore, for $t \geq \alpha(J)$, and as a consequence for $t \geq \alpha(K)-1$,
$$\HP_{K}(t) = \HP_{J}(t) + \# \Delta \stackrel{\text{by } \eqref{eq:HF_J}}{=} \HF_{J}(t) + \# \Delta \geq \HF_{K}(t).$$
The last inequality follows from the fact that
$$\{ \mu \in M(n-1) : \mu \notin K\} = \{ \mu \in M(n-1) : \mu \notin J \} \cup \Delta.$$
This ends the proof in case $c=0$.

 \textbf{Case $c= 1$}. Let $I^{(m)}$ be a sequence of ideals in $\field[\PP^n]$ with a subsequence of depth $n-1$. We
restrict ourselves to this subsequence and denote it by $I^{(m)}$. Moreover, by $\overline{I^{(m)}}$ we denote the image of $I^{(m)}$ in $\field[\PP^{n-1}]$.

Observe that for $t \gg 0$ $$\HP_{\overline{I^{(m)}}}(mt)=\HP_{I^{(m)}}(mt)-\HP_{I^{(m)}}(mt-1).$$

Since $\aHP_I(t)$ is a polynomial by \cite[Theorem 13]{DST}, the derivative  $\aHP_I'(t)$ exists and is equal to the limit of the difference quotient. Moreover for any $T\in\mathbb{R}$ and for $t\in[0,T]$, $\aHP_{I}$ is a limit of uniformly convergent polynomials with bounded degrees. Hence
$$\displaystyle\aHP_I'(t)=\lim_{h\to 0}\frac{\aHP_I(t)-\aHP_I(t-h)}{h}=\lim_{h\to 0}\lim_{m\to\infty}\frac{1}{h} \left(\frac{\HP_{I^{(m)}}(mt)}{m^n}- \frac{\HP_{I^{(m)}}(m(t-h))}{m^n}\right)=$$
$$\lim_{m\to\infty}\frac{1}{\frac 1 m} \left(\frac{\HP_{I^{(m)}}(mt)}{m^n}- \frac{\HP_{I^{(m)}}(m(t-h))}{m^n}\right)= \lim_{m\to\infty} \left(\frac{\HP_{I^{(m)}}(mt)}{m^{n-1}}- \frac{\HP_{I^{(m)}}(m(t-h))}{m^{n-1}}\right)=\aHP_{\overline{I}}(t).$$

We proceed by induction on $c$.
\end{proof}

\section{Examples}
In this section we present some interesting examples. There are
two types of examples.

The first type of examples concerns ``crosses". By a ``cross" we
mean two intersecting lines. We show in Example
\ref{five-crosses}, that our theorem gives better bound than those
possible to compute with help of a computer.

The second type of examples are examples of some star
configurations in $\mathbb{P}^4$. In particular, in Example~\ref{pochodna} we show that it is necessary to take the root of a
derivative of the polynomial $\Lambda_I$, not of the polynomial itself. We also formulate a
problem, which may be viewed as a generalization of Nagata and
Nagata-type conjectures, see \cite{DHST}.

In the sequel we will need the notion of a limiting shape and some
results from \cite{DSST} and \cite{DST}. Let us start with
recalling the notion of limiting shapes.

To define the limiting shape, consider first \emph{generic initial
ideal} $\gin(I)$ of $I$, as the initial ideal, with respect the
degree reverse lexicographical order, of a generic coordinate
change of $I$. Galligo \cite{Gal74} assures that for a homogeneous
ideal $I$ and a generic choice of coordinates, the initial ideal
of $I$ is fixed, hence the definition of $\gin(I)$ is correct.

In the next step consider the sequence of monomial ideals
$\gin\left(I^{(m)}\right)$. The $m$-th symbolic power of a radical ideal $I$
is saturated, hence by Green \cite[Theorem 2.21]{Gre98} no minimal
generator of $\gin\left(I^{(m)}\right)$ contains the last variable $x_{n}$.
Therefore these monomial ideals can be naturally regarded as
ideals in $\field[x_0,\dots,x_{n-1}]$. The Newton polytope of a
monomial ideal is defined as a convex hull of the set of
exponents:
$$P(J) := \conv(\{ \alpha \in \RR^n : x^{\alpha} \in J\}).$$

The limiting shape of an ideal $I$ as above is defined as
$$\Delta(I) = \bigcup_{m=1}^{\infty} \frac{P\left(\gin\left(I^{(m)}\right)\right)}{m},$$
(see Mayes \cite{Maypoints}).

Define $ \Gamma_I$ as the closure of the complement of $\Delta(I)$
in $\RR^n_{\geq 0}$.

Theorem 3 in \cite{DST} says that for radical LBSR ideal $I$ and
for $\TT_t = \{ (x_1,\dots,x_n) : x_1 + \ldots + x_n \leq t \}$ we
have
\begin{equation}\label{3inDST}
\aHP_{I}(t) = \vol (\Gamma_I\cap \TT_t), \quad t \gg 0.
\end{equation}

Now we present the first type of examples.
 First, take a ``cross", ie two intersecting lines, in $\mathbb{P}^3$.
A cross is a complete intersection of type~$(2,1)$. From the
results of Mayes \cite{Mayci}, (Theorem 3.1) we have that for the
ideal $I$ of a cross
$$\gin(I^m)=T(m)\times\mathbb{R},$$
where $T(m)$ is a triangle in $\mathbb{R}^2$ with vertices $(0,0),
(m,0), (0,2m)$. From this we have that the asymptotic limiting
shape (see \cite{Mayci} or \cite{DSST})
$$\Gamma_I=T(1)\times \mathbb{R}.$$
Thus we obtain that the asymptotic Hilbert polynomial of a cross in $\mathbb{P}^3$ 
\begin{equation}\label{eq:aHP_cross}
\aHP(t)=t-1.
\end{equation}
Indeed, by equation \eqref{3inDST}  we compute the volume of $\Gamma_I$ cut by the plane $x+y+z=t$, ie
$$\iint_{T(1)}t-(x+y)dxdy.$$

In the examples below we will also need the following, rather obvious
fact: if $Z_I$ and $Z_J$ are two disjoint sets given as zero sets of
radical LBSR  ideals $I$ and $J$ respectively then
\begin{equation}\label{eq:aHP_disjoint}
\aHP_{I\cap J}=\aHP_I+\aHP_J.
\end{equation}
The formula follows from the definitions of the Hilbert polynomial and
the asymptotic Hilbert polynomial, from the exact sequence:
$$0\to J/I\cap J\to R/I\cap J\to (R/I\cap J) / (J/I\cap J) \simeq R/J\to 0$$
and from the fact that for the ideals of disjoint sets
$R/I=(I+J)/I\simeq J/I\cap J$.

Consider  $s$ generic crosses in $\mathbb{P}^3$ with the ideal
$I_s$. Take  the polynomial $\Lambda_s$ (see \eqref{eq:aHP_cross} and \eqref{eq:aHP_disjoint}):
$$\Lambda_s(t)=\frac{t^3}{6}-s(t-1).$$
Denote the largest real root of $\Lambda_s$ by $\gamma_s$. For
$s=2,3,4$ we have that 
$\gamma_2=2.76873...$, $\gamma_3=3.60687...$, $\gamma_4=4.29021...$. As
in the same time $\alpha(I_s)=s$, we have that
$\widehat\alpha(I_s)\leq s$, which is less than the root of~$\Lambda_s$. 

In case $s=5$ the situation is different:
\begin{example}\label{five-crosses} Consider  5 generic crosses in $\mathbb{P}^3$ with the ideal
$I_5$.
Here $\Lambda_5(t)=\frac{t^3}{6}-5(t-1)$ and
$\alpha(I_5)=5$, moreover, using a computer we may check that for
$m=2,\ldots,10$ we still have $\alpha(I_5^{(m)})=5m$ (and then the
time of computations grows rapidly), but from our theorem we
immediately know that $\widehat\alpha(I_5)\leq \gamma_5=4.88447...$.
\end{example}

\begin{remark}
We may want to compute (with help of a computer) the expected
initial degree $e\alpha_m$ of $I^{(m)}$ (in the similar way as it was
done for fat flats in \cite{DHST}). This could be the way to bound
$\widehat\alpha$ by finding the lowest possible term (or infimum)
of $\frac{e\alpha_m}{m}$. There are two problems with this method.
The first is that the expected degree $e\alpha_m$ may go down very
slowly. For five crosses $e\alpha_m=5$ up to $m=12$ and
$\frac{e\alpha_{13}}{13}=\frac{64}{13}=4,92307...> 4.88447...$.

The second problem is more important. We do not know if the
expected initial degree is properly computed. The formula for a dimension  of a system of forms of degree $d$ vanishing along a given set with multiplicity
$m$ may be correct only for $d$ big enough. We have an unpublished result that for a cross in $\mathbb{P}^3$ the formula is correct for $d\geq 2m-2$.
\end{remark}


Now we move to the second type of examples, concerning  star
configurations.

We begin with giving an experimentally found  formula for the
asymptotic Hilbert polynomial of a star configurations given in
$\mathbb{P}^n$ by intersecting every $c$ out of $s$ generic
hyperplanes. We will denote the ideal of  such a configuration by
$I_{c,s,n}$. From Theorem 1.1 in \cite{DSST} we know that
$\Gamma_{I_{c,s,n}}=\Gamma_{I_{c,s,c}}\times \mathbb{R}^{n-c}$,
where $\Gamma_{c,n,c}$ is a simplex in $\mathbb{R}^c$ with
vertices $\frac{s}{c},\frac{s-1}{c-1},\ldots,s-(c-1).$


Using equation \eqref{3inDST} we see that to compute
$\aHP_{I_{c,s,n}}(t)$ it is enough to compute the volume of
$\Gamma_{I_{c,s,n}}$ cut by the plane $x_1+\ldots+x_n=t$.

Denote
$$a_1=\frac{s}{c},\ a_2=\frac{s-1}{c-1},\ldots,a_c=s-(c-1).$$ The
volume is the integral:
$$\int\limits_0^{a_1}dx_1\int\limits_0^{-\frac{a_2x_1}{a_1}+a_2}dx_2\ldots\int\limits_0^{-\frac{a_cx_1}{a_1}-\ldots-\frac{a_cx_{c-1}}{a_{c-1}}+a_c}dx_c
\int\limits_0^{t-x_1-\ldots-x_c}dx_{c+1}\ldots\int\limits_0^{t-x_1-\ldots-x_{n-2}}(t-x_1-\ldots-x_{n-1})dx_{n-1}.$$

By computing the integral for small values of $n$ and $c$ we found
the formula:
$$\aHP_{I_{c,s,n}}(t)=$$
 $$a_1\cdot a_2\cdots a_c\frac{(n-c)!}{n!}
\left(\binom{n}{0}(-1)^{n-c}\sum^{n-c}+\binom{n}{1}(-1)^{n-c-1}\sum^{n-c-1}\cdot
t+\binom{n}{2}(-1)^{n-c-2}\sum^{n-c-2}\cdot t^2+\right.$$
$$\left. +\ldots+\binom{n}{n-c}(-1)^{0}\sum^{0}\cdot t^{n-c}\right),$$
where the sum $\displaystyle{\sum^k}$ denotes the sum of all
monomials of degree $k$ in variables $a_i$. So far we are not able
to prove the formula.

The next, important example shows  that it is necessary to take an appropriate derivative of the polynomial $\Lambda_I$.
\begin{example}\label{pochodna}
 Take star configurations of lines in $\mathbb{P}^4$, given by
$s$-hyperplanes, $s\geq 4$, with the ideal $I_{3,s,4}$. From the considerations  above we have that
$$\Lambda_{I_{3,s,4}}(t)=\frac{t^4}{24}-\aHP_{I_{3,s,4}}(t)=\frac{t^4}{24}+\frac{1}{864}(-30 s + 67 s^2 - 48 s^3 + 11 s^4) +
\frac{1}{864} (-48 s + 72 s^2 - 24 s^3) t.$$ It is easy to see (with computer's help) that the polynomial $\Lambda_{I_{3,s,4}}(t)$ has no real zeros for
$s\geq 4$. Indeed, the value at the zero of the derivative (minimum point) is greater than  $0.012s^4 $ for $s\geq 10$, and for  $s=4,\ldots,9$ we see by a
direct check that the value at the minimum point is positive. 

The derivative of $\Lambda_{I_{3,s,4}}(t)$ has a root. This root is approximately equal
$\frac{s}{\sqrt[3]{6}}$. The value of $\widehat\alpha(I_{3,s,4})$ equals $\frac{s}{3}$ (see Example 8.3.4 in \cite{primer}).
\end{example}

\begin{remark}\label{o-fat-flatsach}
Theorem 2.5 in \cite{DHST} gives a bound for the Waldschmidt
constant of the ideal of the disjoint sum of linear subspaces
(flats) in $\mathbb{P}^n$. Namely, it says that:
\begin{theorem}\label{zpracyzBrianem}
Let $ n, r, s$ be integers with $n \geq 2r + 1, r \geq  0$ and $s
\geq  1$. Let $I$ be the ideal of $s$ disjoint $r$-planes in
$\mathbb{P}^n$. Then the polynomial $\Lambda_I(t )$ has a single
real root bigger than or equal to $1$. Denote this largest real
root by $\gamma_I$. Then  $\widehat\alpha(I)\leq \gamma_I$.
\end{theorem}
In particular the theorem holds for one linear subspace of
codimension $n-r$. Observe, that it is in a sense accidental that
in this case $\widehat\alpha(I)$ is bounded from above by the root of
$\Lambda_I$, as according to Theorem~\ref{main} we should take the
largest root of the $r$-th derivative of the polynomial.
\end{remark}

Next, we formulate a problem which may be viewed as a generalization of Nagata-type conjectures (see \cite{DHST}). Note, that one may generalize the conjecture of Nagata asking if there exists a number $N_0$ (depending on an algebraic variety $X$), such that for $s\geq N_0$ the Waldschmidt constant of the ideal of $s$ generic
points on $X$ is maximal possible. The original Nagata conjecture says that $N_0=10$ for $\mathbb{P}^2$, and the maximal possible value of $\widehat\alpha$ is
$\sqrt{s}$, ie the largest root of the polynomial $\Lambda$ for these points.

\begin{problem}
Let $X$ be an algebraic variety. Take a radical ideal $I$ in
$\mathbb{K}[X]$. Take the ideal
$$J=\bigcap_{j=1}^s \phi_j(V(I)),$$
where $\phi_j, j=1,\ldots,s$, is a generic change of coordinates.
Then, for $s$ big enough, the Waldschmidt constant of $J$ is
maximal possible, ie equal to the largest root of the suitable
derivative of the polynomial $\Lambda_J(t)=\frac{t^n}{n!}-s\cdot\aHP_I(t)$.
\end{problem}

\textit{Acknowledgements. The first and third authors research was partially supported by the National Science Centre, Poland, grant 2014/15/B/ST1/02197.}

\bigskip \small

\bigskip
   Marcin Dumnicki, {\L}ucja Farnik, Halszka Tutaj-Gasi\'nska,
   Jagiellonian University, Faculty of Mathematics and Computer Science, {\L}ojasiewicza 6, PL-30-348 Krak\'ow, Poland

\nopagebreak
   \textit{E-mail address:} \texttt{Marcin.Dumnicki@uj.edu.pl}

   \textit{E-mail address:} \texttt{Lucja.Farnik@uj.edu.pl}

   \textit{E-mail address:} \texttt{Halszka.Tutaj-Gasinska@uj.edu.pl}

\end{document}